\documentclass[12pt]{extarticle}
\usepackage{amsmath, amsthm, amssymb, color}
\usepackage{graphicx}
\usepackage{caption}
\usepackage{mathtools}
\usepackage{enumitem}
\usepackage{verbatim}
\usepackage{longtable}
\usepackage{pifont}
\usepackage{makecell}
\usepackage{tikz}
\usetikzlibrary{matrix}
\usetikzlibrary{arrows}
\usepackage[normalem]{ulem}
\usepackage{subcaption}
\usepackage{xcolor}
\usepackage{tikz-cd}
\usepackage{url}

\usepackage[
backend=biber,
style=numeric
]{biblatex}
\usepackage[colorlinks,plainpages,hypertexnames=false,plainpages=false]{hyperref}
\hypersetup{urlcolor=black, citecolor=black, linkcolor=black}

\oddsidemargin \evensidemargin
\usepackage{makecell}
\usepackage{multirow,array}

\newtheorem*{zz*}{Zu zeigen} 
\newtheorem{theorem}{Theorem}
\newtheorem*{satz*}{Satz}
\newtheorem{proposition}[theorem]{Proposition}
\newtheorem{lemma}[theorem]{Lemma}
\newtheorem{corollary}[theorem]{Corollary}

\theoremstyle{definition}
\newtheorem{definition}[theorem]{Definition}

\newtheorem{remark}[theorem]{Remark}

\newtheorem{conventions}[theorem]{Conventions}
\newtheorem{question}[theorem]{Question}
\newtheorem*{claim*}{Claim}
\numberwithin{theorem}{section}
\newtheorem{innercustomthm}{Theorem}
\newenvironment{customthm}[1]
{\renewcommand\theinnercustomthm{#1}\innercustomthm}
{\endinnercustomthm}

\newcommand{\PP}{\mathbb{P}}

\newcommand{\CC}{\mathbb{C} }
\newcommand{\ZZ}{\mathbb{Z}}
\newcommand{\NN}{\mathbb{N}}

\newcommand{\Aff}{\mathbb{A}}
\newcommand{\Pic}{\operatorname{Pic}}

\newcommand{\cA}{\mathcal{A}}
\newcommand{\cB}{\mathcal{B}}

\newcommand{\OO}{\mathcal{O}}
\newcommand{\cH}{\mathcal{H}}
\newcommand{\cE}{\mathcal{E}}

\newcommand{\cF}{\mathcal{F}}
\newcommand{\cV}{\mathcal{V}}

\newcommand{\cW}{\mathcal{W}}

\newcommand{\op}{\operatorname}

\DeclareMathOperator{\End}{\mathcal{E}nd}

\renewcommand{\tilde}{\widetilde}		

\let\div\relax							
\DeclareMathOperator{\div}{div}

\newcommand{\thistheoremname}{}
\newtheorem*{genericthm*}{\thistheoremname}
\newenvironment{namedthm*}[1]
{\renewcommand{\thistheoremname}{#1}%
	\begin{genericthm*}}
	{\end{genericthm*}}
\newenvironment{namedtheorem*}[1]
{\renewcommand{\thistheoremname}{#1}%
	\begin{genericthm*}}
	{\end{genericthm*}}

\title{The restricted Hitchin map of wobbly vector bundles}
\author{Clemens Nollau}
\date{}
\begin{document}
	\maketitle
	\begin{abstract}
		This article studies the restricted Hitchin map $h_V$ of a stable rank 2 vector bundle
		$V$ on a smooth projective curve $C$. This map associates to a trace-free twisted endomorphism
		$\varphi : V \to V \otimes K_C$ its determinant, which is a quadratic differential. We show that if $V$ is a general wobbly vector bundle, then it has a single nilpotent twisted endomorphism, up to scalars. As a consequence we show that $h_V$ is generically finite and we compute its degree. Furthermore, we show that if $C$ is not hyperelliptic, then the image of $h_V$ contains a quadratic differential with simple zeros. This is equivalent to saying that there is a smooth spectral curve associated to a twisted endomorphism of $V$.
	\end{abstract}
	\section{Introduction} 
	 Let $C$ be a smooth projective curve of genus $g \geq 2$ over $\CC$ and let $V$ be a stable vector bundle of rank $2$ on $C$. The space of trace-free twisted endomorphisms of $V$ is 
	 \begin{equation*}
	 	H^0(C,\End_0(V)\otimes K_C) = \left\{ \phi\in \op{Hom}(V,V\otimes K_C) \mid \operatorname{tr}(\phi) = 0\right\}.
	 \end{equation*}
	  We are interested in the following map: 
	 \begin{equation*}
	 	h_V: H^0(C,\End_0(V)\otimes K_C) \to H^0(C,K_C^2), \quad\phi \mapsto \det \phi.
	 \end{equation*}
	 It is called the \textit{restricted Hitchin map}.
	 The source and the target both have dimension $3g-3$. Hence, it is natural to ask the following. 
	 \begin{question}\label{question}
	 	Assume $V$ is a stable vector bundle. Is $h_V$ dominant?
	 \end{question}
	 We remark that the same question can be asked for any vector bundle of rank $2$ but as noticed by Will Sawin \cite{willsawin} dominance can already fail for semistable vector bundles: If $C$ is a hyperelliptic curve of genus $g\geq 3$ and $V = \OO_C^{\oplus 2}$, its restricted Hitchin map is not dominant because the multiplication map $\op{Sym}^2H^0(K_C)\to H^0(K_C^2)$ is not surjective.
	 
	 Let us indicate the origin of the restricted Hitchin map: Consider $\op{Bun}^{s}_{2,\Lambda}(C)$ the moduli space of stable vector bundles on $C$ of rank $2$ with determinant $\Lambda$. By deformation theory the space of trace-free twisted endomorphism of $V$ is the cotangent space of 
 $\op{Bun}^s_{2,\Lambda}(C)$ at $V$. The cotangent bundle $T^\vee\op{Bun}^s_{2,\Lambda}(C)$ can be partially compactified by the moduli space $\cH_C(2,\Lambda)$ of semistable $\op{SL}_2$-Higgs bundles with determinant $\Lambda$. The Hitchin map is then defined as follows:
	 \begin{equation*}
	 	h: \cH_C(2,\Lambda) \to H^0(C,K_C^2),\quad (V,\phi) \mapsto \det \phi.
	 \end{equation*}
	 In \cite{hitchin} Hitchin showed that $\cH_C(2,\Lambda)$ is an algebraic symplectic variety and that $h$ is a Lagrangian fibration. This means $h$ is proper, its fibers are Lagrangian subvarieties and the general fibers are Prym varieties of spectral curves. Let us introduce the spectral curves in question. Consider the total space of the canonical line bundle $\Aff(K_C)$. Let $\tau$ be its tautological section. If $\eta $ is a quadratic differential in $H^0(K_C^2)$, we define the corresponding spectral curve by 
	 \begin{equation*}
	 	\tilde{C} := V(\tau^2 - \eta) \subset \Aff(K_C).
	 \end{equation*}
	 It comes with a double cover $\psi :\tilde{C} \to C$. Now we can state the Beauville-Narasimhan-Ramanan (BNR) correspondence \cite{hitchin},\cite{bnr}:
	 Assume the spectral curve corresponding to the quadratic differential $\eta$ is integral. Then there is a bijection:
	 \begin{align*}
	 	\left\{ \begin{array}{c}
	 		\text{$M$ torsion-free sheaf of rank $1$ on $\tilde{C}$} \\ \text{such that $\det \psi_*M \cong \Lambda$}
	 	\end{array}\right\} &\leftrightarrow h^{-1}(\eta) \\
	 	M &\mapsto (\psi_*M,\phi).
	 \end{align*}
	 Here the twisted endomorphism $\phi$ is induced by multiplication with $\tau : M \to M\otimes \psi^*K_C$.
	 Let us assume that $\tilde{C}$ is smooth. Then the fiber is a generalized Prym variety. 
	 
	 The BNR-correspondence can be applied to the program of extending Brill-Noether theory to classes of algebraic curves which are not general in the moduli space of curves. Here we consider the class of spectral curves $\tilde{C}$ defined above. One challenge is to prove the existence of line bundles on $\tilde{C}$ with many global sections. To do so, one starts with a stable vector bundle $V$ of rank $2$ with large $h^0(V)$. Then there exists a line bundle $M$ on $\tilde{C}$ with $\psi_*M \cong V$ if and only if $h_V^{-1}(\eta)$ is nonempty. Hence, if we can prove that $h_V$  is dominant, such a line bundle exists for a general spectral curve over $C$. This connects the Brill-Theory of vector bundles with the Brill-Noether theory of spectral curves. The author started an investigation in this direction in \cite{nollau}. 
	 Generalizations of the restricted Hitchin map provide insight into the Brill-Noether theory of other classes of curves: 
	 In \cite[Lemma 3.2]{larsonvemulapalli} the dominance of a certain map is proven to deduce non-emptiness of Brill-Noether loci of curves lying on Hirzebruch surfaces.
	 
	  As an indication that the restricted Hitchin map might be dominant for every stable vector bundle of rank $2$ we prove in section \ref{section:smooth} the following:
	 \begin{customthm}{A}\label{thm:A}
	 		\emph{
	 		Let $C$ be a smooth projective curve of genus $g \geq 2$ that is not hyperelliptic and take $V$ to be a stable vector bundle of rank $2$ on $C$. Then there exists a twisted endomorphism $\phi \in H^0(\End_0(V)\otimes K_C)$ whose spectral curve is smooth.}
	 \end{customthm} 
	 
	A second reason to study the restricted Hitchin map is that it gives a better understanding of wobbly vector bundles. They are defined as follows: 
	\begin{definition}
		A stable vector bundle $V$ of rank $2$ is said to be \textit{very stable} if the only nilpotent twisted endomorphism it admits is $0 $. If $V$ is stable but not very stable, it is called \textit{wobbly}.
	\end{definition}
	In \cite{nietopauly} it was shown that the restricted Hitchin map is quasi-finite if and only if $V$ is very stable. Therefore, to answer Question \ref{question} we need to concentrate on wobbly vector bundles. This paper shows that the geometry of $V$ is reflected in the behaviour of $h_V$. We make the following basic observation: If $\phi$ is a nilpotent twisted endomorphism, then its kernel is a subbundle $L$ such that $\op{Hom}(L,V/L\otimes K_C) \neq 0$. We give line subbundles of this type a name: 
	\begin{definition}
		A line subbundle $L$ of $V$ is said to be \textit{bad} if $\op{Hom}(V/L,L\otimes K_C) \neq 0$.
	\end{definition}
	It is clear that  a stable $V$ is wobbly if and only if it has a bad line subbundle. 
	The locus of wobbly vector bundles in $\op{Bun}^s_{2,\Lambda}(C)$ is a divisor in $\op{Bun}^s_{2,\Lambda}(C)$ by \cite[Theorem 1.1]{PalPauly} and its irreducible components are known. To explain this result, we need the following definition:
	\begin{definition}
		The algebraic set consisting of all wobbly vector bundles in $\op{Bun}^s_{2,\Lambda}(C)$ is the \textit{wobbly locus} $\cW$. Define the following stratification of $\cW$: 
		\begin{align*}
			\cW^0_{k} &:= 
			\left\{ V \in \op{Bun}^s_
			{2,\Lambda}(C) \;\middle|\; 
			\begin{array}{l}
				\text{$V$ has a bad line subbundle $M$} \\ \text{ of degree $1 - k$}
			\end{array}
			\right\}\\
			&\text{for $2-\lambda \leq k \leq g-\lambda$}.
		\end{align*}
		We also define $\cW_k$ to be the closure of $\cW_k^0$ in $\cW$.
	\end{definition}
	Let us assume from now on that $\lambda := \deg \Lambda$ is in $\{0,1\}$. This is no restriction of generality because by twisting our vector bundles $V$ with line bundles we can always reach $\deg V \in \{0,1\}$. It was shown in \cite[Theorem 1.1]{PalPauly} that $\cW$ has the following decomposition:
	\begin{equation*}
		\cW = \cW_{\lceil\frac{g-\lambda}{2}\rceil} \cup \ldots \cup \cW_{g-\lambda}.
	\end{equation*}
	Furthermore, all $\cW_k$ appearing above are irreducible, except for $\cW_g$ in the case $\lambda = 0$, which is the union of $2^{2g} $ irreducible components. Let us recall the parametrization of the stratum $\cW_k^0$ used in \cite{PalPauly}. Define the following $Z_k \subset \op{Pic}^{1-k}(C)$ by  
	\begin{equation*}
		Z_k := \left\{M \in \op{Pic}^{1-g}(C) \mid h^0(M^2 \otimes \Lambda^\vee \otimes K_C) > 0\right\}
	\end{equation*}
	These are all line bundles which are a bad line subbundle of some vector bundle in $\op{Bun}_{2,\Lambda}(C)$. By Riemann-Roch we have $\dim \op{Ext}^1(M^\vee\otimes \Lambda,M) = g+2k+\lambda-3$ for all $M \in Z_k$. There exists a vector bundle $\cV_k$ on $Z_k$ whose fiber at $M$ is $\op{Ext}^1(M^\vee\otimes \Lambda,M)$. Consider all extensions $[\xi] \in \PP(\cV_k)$ such that the induced vector bundle $V_\xi$ is stable. They form an open subset of $\PP(\cV)_k$ which we denote by $\PP(\cV_k)^s$. Then there is a surjective map:
	\begin{equation*}
		\phi_k: \PP (\cV_k) \to \cW_k^0, \quad [\xi] \mapsto V_\xi.
	\end{equation*} 
	 
	In section \ref{section:unique_bad} we prove for the general element of each stratum of $\cW$ the following:
	\begin{customthm}{B}\label{thm:B}
		\emph{
			Fix $\Lambda \in \Pic^\lambda(C)$ and $k$ with $2-\lambda \leq k \leq g-\lambda$.
		Let $V$ be a be a general element of $\cW_k^0$. Then $V$ has a unique bad line subbundle.}
	\end{customthm} 
	The proof is based on dimension counting: We will show that the set of all extensions $[\xi] \in\PP(\cV_k)^s$ for which $V_\xi$ admits a second bad line subbundle is a proper subvariety.
	
	The fact that the general wobbly vector bundle $V$ in each irreducible component of $\cW$ has this simple structure makes it possible to analyze $h_V$ in more detail. We do this by considering the rational map 
	\begin{equation*}
			\PP h_V:	\PP H^0(\End_0(V) \otimes K_C) \dashrightarrow \PP H^0(K_C^2), \quad[\phi] \mapsto [\det \phi].
	\end{equation*}
 We show how to resolve the indeterminacy by blowing up twice and we obtain the following:
	\begin{customthm}{C}\label{thm:C}
		\emph{Assume $V$ is a general vector bundle in $\cW_k$, where $\lceil\frac{g-\lambda}{2}\rceil \leq k \leq g-\lambda$. 
		Then the following holds:
		\begin{itemize}
			\item[(i)] If $\eta \in H^0(K_C^2)$ defines an integral spectral curve, then $h_V^{-1}(\eta)$ is finite. In particular, $h_V$ is dominant.
			\item[(ii)] $	\deg h_V = 2^{3g-3} - 2^{2g-2k - \lambda +1}$.
		\end{itemize} }
	\end{customthm}
One can compare this to the degree of $h_V$ when $V$ is very stable. It is shown in \cite[Remark 5.4]{bnr} that $\deg h_V = 2^{3g-3}$. It would be interesting to compute the degree for the general wobbly vector bundles in the remaining strata $\cW_k^0$ and also for wobbly vector bundles of higher rank. The proof of Theorem \ref{thm:C} is split up in two parts. In section \ref{section:blowup} we resolve the indeterminacy of $\PP h_V$ while in section \ref{section:blowup} we analyze the resulting blowup in general.

 Section \ref{section:reducible} is concerned with the fibers of the restricted Hitchin map $h_V^{-1}(\eta)$ when $\eta$ defines a reducible spectral curve. A Higgs bundle $(V,\phi)$ with a reducible spectral curve corresponds roughly to the choice of two line subbundles of $V$. For wobbly vector bundles we see that the fibers $h_V^{-1}(\eta)$ will in general be positive dimensional. This result is probably well known but we include it since it gives qualitative insight into the restricted Hitchin map.  
\subsection*{Acknowledgements} 
I would like to thank my PhD-advisor, Daniele Agostini for encouragement and useful feedback on this project, which greatly improved it a lot. Thanks also to Tony Pantev, Duong Dinh, and Angela Ortega for useful discussions on the topic.
	\section{Preliminaries}
	\begin{conventions}
		All schemes we consider are $\CC$-schemes and unless said otherwise a point of a scheme is understood to be a $\CC$-point.
		Throughout the article we will use Fulton's convention for projectivization: If $U$ is a vector space, the points of $\PP U$ are one-dimensional subspaces of $U$. If $X$ is a scheme and $L$ is a line bundle on it, we write $\phi_L: X\dashrightarrow\PP H^0(X,L)^\vee$ for the rational map induced by the complete linear system of $L$. A subsheaf $U$ of a vector bundle $V$ will be called a subbundle if $V/U$ is a vector bundle. Throughout the text $C$ is a smooth projective curve of genus $g \geq 2$ unless stated otherwise.
	\end{conventions}
	Let $V$ be a vector bundle of rank $2$ on $C$.
	First we compute the differential of the restricted Hitchin map $h_V$. For $\phi \in \op{H}^0(\End_0(V) \otimes K_C)$ we have $\det \phi  = \frac{1}{2}\op{tr}(\phi^2)$. This implies that 
	\begin{equation*}
		d_\phi h_V: H^0(\End_0(V)\otimes K_C) \to H^0(K_C^2),\quad \psi \mapsto \op{tr}(\psi\phi).
	\end{equation*}
	Now we assume $\phi$ is nilpotent with kernel $L_1$. Write $L_2$ for the line bundle $V/L_1 \cong \op{im}(\phi)$. The twisted endomorphism $\phi$ is induced by a morphism $\phi_{12}: L_2 \to L_1\otimes K_C$. Furthermore, $V$ is an extension $\xi$ of $L_2$ by $L_1$. We denote the natural map $H^0(\End_0(V)\otimes K_C) \to \op{Hom}(L_1,L_2\otimes K_C)$ by $c_\xi$. Take $\psi \in H^0(\End_0(V) \otimes K_C)$. Let us compute $\op{tr}(\psi\phi)$: The composition $\psi \phi$ is zero on $L_1$. Therefore, it comes from a map $L_2 \to V\otimes K_C^2$. If we compose $L_2 \to V\otimes K_C^2$ with $V \otimes K_C^2 \to L_2\otimes K_C^2$, we get multiplication by $c_\xi(\psi)\cdot \phi_{12}$. Hence, we obtain:
	\begin{equation}\label{eq:diff_nilpotent}
		d_\phi h_V(\psi) = \op{tr}(\psi\phi) =  \phi_{12} c_\xi(\psi).
	\end{equation}
	This shows the following:
	\begin{lemma}\label{lemma:ker_diff}
		Assume $V$ is a vector bundle of rank $2$ and $\phi$ is a nonzero nilpotent twisted endomorphism of $V$. Then the kernel of $d_\phi h_V$ is given by $\ker c_\xi$. 
	\end{lemma}
	We can describe the image of $c_\xi$ using the cup product map $\xi \cup (-): H^0(L_1^\vee\otimes L_2\otimes K_C) \to H^1(K_C)$.
	\begin{lemma}\label{lemma:image_c}
		Assume $V$ is a stable vector bundle of rank $2$ obtained from an extension $\xi$ of line bundle $L_2$ by a line bundle $L_1$. Assume furthermore that $h^0(L_1^\vee \otimes V) = 1$. Then 
		\begin{equation*}
			\op{im}(c_\xi) = \ker(\xi \cup (-))
		\end{equation*}
		holds. In particular, we have $\dim \op{im}(c_\xi) = h^0(L_1^\vee\otimes L_2\otimes K_C) -1$.
	\end{lemma}
	\begin{proof}
		See \cite[equation (2.5)]{dinh_teschner}.
	\end{proof}
	\begin{lemma}\label{lemma:reducible}
		Let $V$ be as in Lemma \ref{lemma:ker_diff}. Then the spectral curve of any $\phi$ in $\ker(c_\xi)$ is not integral. 
	\end{lemma}
	\begin{proof}
		By assumption we have $\phi(L_1) \subseteq L_1\otimes K_C$. Let $\phi_{11}\in H^0(K_C)$ be the induced section. One can then check that $\det \phi = -\phi_{11}^2$ and the defining equation of the spectral curve is $(\tau - \phi_{11})(\tau + \phi_{11 })$.
	\end{proof}
	In section \ref{section:wobbly} we will use the projectivization of the restricted Hitchin map. This is the following rational map:
	\begin{equation*}
		\PP h_V: \PP H^0(\End_0(V)\otimes K_C) \dashrightarrow \PP H^0(K_C^2).
	\end{equation*}
	The indeterminacy locus of $\PP h_V$ consists of all elements $[\phi]$ where $\phi$ is a nilpotent twisted endomorphism. Every such $\phi$ has a one-dimensional kernel $L$ and $\phi$ induces a map $V/L \to L\otimes K_C$. Hence $L$ is a bad line subbundle. Conversely if $L$ is bad, every element of $\op{Hom}(V/L,L\otimes K_C)$ induces a nilpotent twisted endomorphism. Therefore the indeterminacy locus equals:
	\begin{equation}\label{eq:base_locus}
	 \bigcup_{\substack{L\hookrightarrow V: \\
				\text{
					bad line subbundle}}} \PP\left(\op{Hom}(V/L,L\otimes K_C)\right).
	\end{equation}
	We will see in section \ref{section:wobbly} that to resolve the indeterminacy of $\PP h_V$ it does not suffice to blowup this locus with its reduced scheme structure in general.
	\subsection{Hecke modifications}
In the proof of Theorem \ref{thm:A} we will use Hecke modifications of vector bundles. We recall their definition here:
\begin{definition}
	Let $W$ be a vector bundle of rank $2$ on $C$. Then an \textit{Hecke modification} of $W$ at $q \in C$ is an exact sequence of the following form:
	\begin{equation*}
		\begin{tikzcd}
			0 \ar[r]& V´\ar[r]&  W \ar[r]& \OO_q \ar[r]& 0.
		\end{tikzcd}
	\end{equation*}
	The sheaf $V$ is also a vector bundle of rank $2$. 
	The image of $V$ in $W_{|q}$ is generated by $w \in W_{|q}$. We say the Hecke modification is induced by $w$.
\end{definition}
We need the following two facts about Hecke modifications:
	\begin{lemma}\label{lemma:hecke}
		Let $L_1$ and $L_2$ be line bundles on $C$. Assume $V$ is a Hecke modification of $L_1\oplus L_2$ at the point $q \in C$ induced by $w = \begin{pmatrix}
		w_1 \\ w_2
		\end{pmatrix}$. Then the following holds:
		\begin{itemize}
			\item[(i)] If $w_1 = 0$, then $V \cong L_1(-q) \oplus L_2$. 
			\item[(ii)] If $L_1 \cong L_2$, then $V \cong L_1(-q) \oplus L_1$.
		\end{itemize}
	\end{lemma}
	\begin{proof}
		(i) holds because the kernel of $L_1 \twoheadrightarrow \OO_q$ is $L_1(-q)$. For (ii) note that the automorphism group of $L_1\oplus L_2$ is $\op{GL}_2(\CC)$ and it acts transitively on the fiber of $L_1^{\oplus 2}$ at $q$.
		Therefore, every Hecke modification of $L_1^{\oplus 2}$ at $q$ is isomorphic to $L_1(-q)\oplus L_1$.
	\end{proof}
	\section{Higgs bundles with smooth spectral curves}\label{section:smooth}
	In this section we prove that every stable vector bundle $V$ of rank $2$ admits a smooth spectral curve. This result and its proof are inspired by \cite[Proposition 3.4]{hitchin} where all vector bundles of rank $2$ with an integral spectral curve are determined. The technical part of the proof of Theorem \ref{thm:A} is the analysis of the evaluation map 
	\begin{equation*}
		\op{ev}_q: 	H^0(\End_0(V)\otimes K_C) \to H^0(\End_0(V)\otimes K_C \otimes \OO_q).
	\end{equation*}
	at points $q$ of $C$. We will be in the following situation: Take $\phi_0 \in H^0(\End_0\otimes K_{C|q})$ and $\rho \in H^0(\End_0(V)\otimes \OO_C(q))$. Then the composition $\phi_0\rho$ is in the fiber of $\End_0(V)\otimes K_C(q)$ at $q$ and $\op{tr}(\phi_0\rho)$ is in $H^0(K_C(q)_{|q})$. This element has a well defined residue $\op{res}_q \op{tr}(\phi_0\rho)$.  
	\begin{lemma}\label{lemma:local_twisted}
		Let $V$ be a stable rank $2$ vector bundle on $C$ and $q$ be a point of $C$. 
		Then the image of $\op{ev}_q$ is as follows: 
		\begin{equation}\label{eq:residue_van}
			\op{im}(\op{ev}_q) = \left\{\phi_0 \in H^0(\End_0(V) \otimes K_C) \;\middle|\;
			\begin{array}{l}
				\op{res}_q\op{tr}(\phi_0 \rho)  = 0 \\
				\text{for all $\rho \in H^0(\End_0(V)\otimes \OO_C(q))$}
			\end{array}\right\}.
		\end{equation}.
	\end{lemma}
	\begin{proof}
		Consider the short exact sequence:
		\begin{equation*}
			0 \to \End_0(V)\otimes K_C(-q) \to \End_0(V) \otimes K_C\to \End_0(V)\otimes K_{C|q} \to 0.
		\end{equation*}
		Its long exact sequence in cohomology is:
		\begin{equation*}
			\begin{tikzcd}
				0 \ar[r] &H^0(\End_0(V)\otimes K_C(-q)) \ar[r] &  H^0(\End_0(V)\otimes K_C) \ar[r,"\op{ev}_q"] & {} \\
				{} \ar[r]& H^0(\End_0(V)\otimes K_{C|q}) \ar[r,"\delta"] & H^1(\End_0(V)\otimes K_C(-q)) \ar[r]& 0
			\end{tikzcd}
		\end{equation*}
		Here we used the stability of $V$ and Serre duality to conclude that $h^0(\End_0(V)) = h^1(\End_0(V)\otimes K_C) = 0$.
		From the exact sequence we deduce $\op{im}(\op{ev}_q) = \ker(\delta)$. By Serre duality there is a nondegenerate pairing $\langle-,-\rangle$ between $H^1(\End_0(V)\otimes K_C(-q)))$ and $H^0(\End_0(V)\otimes \OO_C(q))$. The cohomology class $\delta(\phi_0)$ is zero if and only if $\langle \delta(\phi_0),\rho\rangle = 0$ for all $\rho \in H^0(\End_0(V)\otimes \OO_C(q))$. To make this pairing explicit we use that the cohomology of a coherent sheaf $\cF$ on $C$ can be computed with the following complex: 
		\begin{equation*}
			\Gamma(C-p,\cF) \oplus \Gamma(\op{Spec}\hat{\OO}_{C,q},\cF) \to \Gamma(\op{Spec}\op{Frac}(\hat{\OO}_{C,q}),\cF).
		\end{equation*}
		 Let $\psi$ be in $\Gamma(\op{Spec}\op{Frac}(\hat{\OO}_{C,q}),\End_0(V)\otimes K_C(-q))$ and write $[\psi]$ for the corresponding cohomology class.  Take $\rho \in H^0(\End_0(V)\otimes \OO_C(q))$. The duality pairing is given by 
		\begin{equation*}
			\langle \rho,[\psi]\rangle =  \op{res}_q \op{tr}(\rho\psi).
		\end{equation*}
		If $\phi_0$ is in $H^0(\End_0(V)\otimes K_{C}\otimes \OO_q)$, we can represent $\delta(\phi_0)$ by any lift of $\phi_0$ to a section $\phi$ of $\Gamma(\op{Spec}\op{Frac}(\hat{\OO}_{C,q}),\End_0(V)\otimes K_C(-q))$. Then we have 
		\begin{equation*}
			\langle \rho,\delta(\phi_0)\rangle = \op{res}_q(\rho\phi) = \op{res}_q(\rho\phi_0).
		\end{equation*}
		Hence, $\phi_0$ is in $\op{im}(\op{ev}_q)$ if and only if the right hand side vanishes for all $\rho \in H^0(\End_0(V)\otimes \OO_C(q))$.
	\end{proof}
	\begin{theorem}
		Let $C$ be a smooth projective curve of genus $g \geq 2$ that is not hyperelliptic and take $V$ to be a stable vector bundle of rank $2$ on $C$. Then there exists a twisted endomorphism $\phi \in H^0(\End_0(V)\otimes K_C)$ whose spectral curve is smooth.
	\end{theorem}
	\begin{proof}
	 We recall the interpretation of $H^0(\End_0(V)\otimes K_C)$ as a linear system on $\PP(V)$ from \cite{hitchin}. Let $\pi : \PP (V) \to C$ be the projection. Then we have
		\begin{equation*}
			\End_0(V)\otimes K_C  \cong \pi_*(\OO(2) \otimes \pi^*(\det V\otimes K_C)).
		\end{equation*} 
		This gives an isomorphism $\tau: H^0(C,\End_0(V)\otimes K_C) \to H^0(\PP (V  ),\OO(2)\otimes \pi^*(\det V\otimes K_C))$. For twisted endomorphism $\phi$ the zero locus of $\tau(\phi)$ consists $[v] \in \PP(V_{|q})$, where $q\in C$, which satisfy 
		\begin{equation*}
			\phi(v) \wedge v = 0.
		\end{equation*}  
		The last equation holds in $(\det V\otimes K_C)_{|q}$. In other words: $v$ is an eigenvector of $\phi$.  
		
		 We will show that $\OO(2)\otimes \pi^*(\det V\otimes K_C)$ has only finitely many base points. Furthermore, for each base point we will find a section $s \in H^0(\OO(2)\otimes \pi^*(\det V\otimes K_C))$ such that $V(s)$ is smooth at this point. Then, Bertini's theorem \cite[Theorem 6.3]{jouanolou} implies that there is $\phi \in H^0(\End_0(V)\otimes K_C)$ such that $V(\tau(\phi))$ is smooth. Recall that in this case $V(\tau(\phi))$ is isomorphic to the spectral curve $\tilde{C}$. Hence, the Theorem follows.
		
		To understand the base points of the linear system we determine the image of the evaluation maps $\op{ev}_q$ for all $q\in C$. We use the description of $\op{im} (\op{ev}_q)$ from Lemma \ref{lemma:local_twisted}. If $H^0(\End_0(V)\otimes \OO_C(q))$ vanished for all $q \in C$, then all evaluation maps would be surjective.  Hence, $\End_0(V)\otimes K_C$ would be globally generated and the same would hold for $\OO(2)\otimes \pi^*(K_C\otimes \det V)$. Therefore, we assume from now on  that there is $q \in C$ such that a nonzero $\rho \in H^0(\End_0(V)\otimes \OO_C(q))$ exists. We distinguish whether $\rho$ is nilpotent or not. 
		
		\textit{Case: $\rho$ is not nilpotent.x} Consider the following variant of the restricted Hitchin map:
		\begin{equation*}
			H^0(\End_0(V)\otimes \OO_C(q)) \to H^0(\OO_C(2q)),\quad \phi \mapsto \det \phi.
		\end{equation*}
		Because $C$ is not hyperelliptic, we have $h^0(\OO_C(2q)) = 1$. Therefore $\det \phi \in \CC^\times$. Hence, we can factor the characteristic polynomial of $\rho$ as $\chi_\rho(T) = (T-\mu)(T+\mu)$ with $\mu \in \CC^\times$. Define $L_\pm = \ker(\rho-\pm \mu \iota)$ where $\iota$ is the natural inclusion $V\hookrightarrow V(q)$. These kernels are line subbundles. By stability we have $\deg L_\pm < \frac{d}{2}$. On the other hand we can use the stability of $V(q)$:
		\begin{equation*}
			\deg L_\pm = d - \deg \op{im}(\rho -\pm \mu\iota) > d - \frac{d+2}{2} = \frac{d}{2} -1.
		\end{equation*} 
		This implies that $d = 2k+1$ and $\deg L_\pm = k$ for a certain $k$. Note that $(\rho - \mu\iota)(\rho + \mu\iota) = 0$. Hence,
		\begin{equation*}
			\op{im}(\rho \pm \mu\iota) = L_\pm(q).
		\end{equation*}
		 Because the line subbundles $L_\pm$ are distinct we obtain two inclusions:
		\begin{equation*}
			L_+ \oplus L_- \xhookrightarrow{\begin{pmatrix}
					i_{L_+}\\ i_{L_-}
			\end{pmatrix}} V \xhookrightarrow{(\rho+\mu\iota,\rho - \mu\iota)}  L_+(q)\oplus L_-(q).
		\end{equation*}
		 In other words, $V$ is a Hecke modification of $L_+(q) \oplus L_-(q)$. The Hecke modification is induced by some $w = \begin{pmatrix}
			w_+ \\ w_-
		\end{pmatrix} \in\left(L_+(q)\oplus L_-(q)\right)_{|q}$. This means a section $s$ of $L_+(q)\oplus L_-(q)$ is a section of $V$ if and only if $s_{|q} \in \langle w\rangle$. The stability of $V$ and Lemma \ref{lemma:hecke} (i) imply that $w_+$ and $w_-$ are both nonzero and $L_+ \ncong L_-$.
		
		 Now we show that $\rho$ generates $H^0(\End_0(V) \otimes \OO_C(q))$. Let us abbreviate $L_+\oplus L_-$ by $W$. By construction $\rho$ is the restriction of the following endomorphism:
		\begin{equation}\label{eq:matrix}
			\begin{pmatrix}
				\mu & 0\\
				0 & -\mu
			\end{pmatrix}: W(q) \to W(q).
		\end{equation}
		From the following diagram one deduces that every $\OO_C(q)$-twisted endomorphism of $V$ induces an $\OO_C(2q)$-twisted endomorphism of $W = L_+\oplus L_-$: 
		\begin{equation*}
			\begin{tikzcd}
				W \arrow[r] & V \arrow[d]& \\
				& V(q) \arrow[r] & W(2q).
			\end{tikzcd}
		\end{equation*}
		This gives an inclusion $H^0(\End_0(V)\otimes \OO_C(q))\hookrightarrow H^0(\End_0(V)\otimes \OO_C(2q))$.
		Now let $\rho'$ be an element of $H^0(\End_0(W)\otimes \OO_C(q))$ induced by an element of $H^0(\End_0(V)\otimes \OO_C(q))$. We will show that $\rho'\otimes \op{id}_{\OO_C(q)}$ is a multiple of the morphism in (\ref{eq:matrix}). We write 
		\begin{equation*}
			\rho' = \begin{pmatrix}
				\rho_{11}' & \rho_{12}' \\
				\rho_{21}' & -\rho_{11}'
			\end{pmatrix}.
		\end{equation*}
		with $\rho_{11}' \in H^0(\OO_C(2q)), \rho_{12}' \in H^0(L_-^\vee\otimes L_+(2q))$ and $\rho_{21}' \in H^0(L_+^\vee\otimes L_-(2q))$. Note that $\rho'_{11|q} = 0$ because $H^0(\OO_C(2q)) \cong H^0(\OO_C(q))$. Here we use that $C$ is not hyperelliptic. Furthermore, we have $0 = \det \rho_{|q}' = \rho'_{12|q}\rho'_{21|q}$. Therefore, at least one of $\rho'_{12|q}$ or $\rho'_{21|q}$ is zero. We observe that $w \in \ker \rho'_{|q}$:
		\begin{equation*}
			\rho'(w) \in V(q)_{|q} \Rightarrow \rho_{|q}'(w) = 0 \in W(3q)_{|q}.
		\end{equation*}
		 If only $\rho_{12|q}'$ was  zero, $w$ would not be in the kernel of $\rho'_{|q}$ because $w_+ \neq 0$. The same holds for $\rho'_{21}$. We have shown $\rho'_{|q} = 0$. Therefore our twisted endomorphism in $H^0(\End_0(V) \otimes \OO_C(q))$ is the restriction of
		 \begin{equation*}
		 	\rho'' := \rho'\otimes \op{id}_{\OO_C(q)}: W(q) \to W(2q).
		 \end{equation*}
		 We observe that $w$ needs to be an eigenvector of $\rho''_{|q}$.  From $w_+ \neq 0 $ and $w_-\neq 0$ we deduce that $\rho''_{21|q}$ is nonzero if and only if $\rho''_{12|q}$ is nonzero. Assume for contradiction that $\rho''_{21|q}$ and $\rho''_{12|q}$ are both not zero. Then $L_+^\vee\otimes L_-(q)$ and $L_-^\vee\otimes L_+(q)$ are both effective. In other words we find $p_1,p_2\in C$ such that 
		\begin{equation*}
			L_+^\vee\otimes L_-(q) \cong \OO(p_1),\quad L_-^\vee\otimes L_+(q) \cong \OO(p_2).
		\end{equation*}    
		Take the tensor product of both equations to get $\OO(p_1+p_2) \cong \OO(2q)$. Because $C$ is not hyperelliptic, we have $q = p_1 = p_2$ and $L_+ \cong L_-$. Then Lemma \ref{lemma:hecke} (ii) shows that $V$ is unstable. This is a contradiction. Therefore, $\rho''_{12|q} = \rho''_{21|q} = 0$. Another application of $L_+ \ncong L_-$ shows $\rho''_{12} = \rho''_{21} = 0$. It follows that $\rho'$ is equal to a multiple of $\begin{pmatrix}
			\mu & 0 \\
			0 & -\mu
		\end{pmatrix}$. We have shown $H^0(\End_0(V) \otimes \OO_C(q)) = \langle \rho \rangle$.
		
		\textit{Case: $\rho$ is nilpotent.} Let $L$ be the kernel of $\rho$. By the same argument we used to compute the degree of $L_+$, we find $d = 2k+1$ and $\deg L = k$ for some $k$. There is a nonzero map $V/L \to L(q)$. From $\deg V/L =L(q)$, we deduce $V/L \cong L(q)$. The vector bundle $V$ is an extension of $L(q)$ by $L$. We are going to show that $H^0(\End_0(V) \otimes \OO_C(q))$ is generated by $\rho$. Take $\rho' \in H^0(\End_0(V) \otimes \OO_C(q)) $. Let $M = \ker \rho'$. Then $V$ is an extension of $M(q)$ by $M$. 
		
		We show that $M \cong L$. Assume for contradiction that $M$ is not isomorphic to $L$. 
		We get a nonzero map $M \to V \to L(q)$. Therefore, $M$ is isomorphic to $L(p-q)$ for some $p \in C$. On the other hand we have $L^2(q) \cong \det V \cong M^2(q)$.
		We deduce $\OO(2p-2q)\cong \OO_C$. Because $C$ is not hyperelliptic we find $q = p$ and $M \cong L$. This is a contradiction and the line bundles are isomorphic. If $L \to V$ and $M \to V$ were different subbundles, $V$ would be a Hecke modification of $L(q)^{\oplus 2}$. Hence, $\ker \rho = \ker \rho'$ and $\rho'$ is a multiple of $\rho$.
		
		Now we can show that there are only finitely many $q$ such that $H^0(\End_0(V)\otimes \OO_C(q)) \neq 0$ holds. Assume that there exists one such point $q$. Then we have seen that $V$ is
	 an extension of a line bundle of degree $k+1$ by a line bundle of degree $k$. By stability $L$ is a maximal line subbundle. It is shown in \cite[Proposition 4.2]{langenarasimhan} that in this case $V$ has only finitely many maximal subbundles. Now if $H^0(\End_0(V) \otimes K_C)$ contains $\rho$ which is not nilpotent we can recover $q$ from the maximal line subbundles $L_+$ and $L_-$: The determinant of $V$ equals $L_+\otimes L_-(q)$. Similarly, if $\rho$ is nilpotent, then we can recover $q$ from $L = \ker \rho$: $\det V = L^2(q)$. Hence, the sheaf $\End_0(V)$ has only finitely many points at which it is not globally generated.
				
		Let us determine the base points of $\OO(2) \otimes \pi^*(\det (V) \otimes K_C)$. First observe that all base points need to lie in $\pi^{-1}(q)$ for a $q$ as above. If $\rho \in H^0(\End_0(V)\otimes \OO_C(q))$ is not nilpotent, then there are no base points: Take $[v] \in \pi^{-1}(q)$. We can find $f \in H^0(\End_0(V) \otimes K_C\otimes \OO_q)$ such that
		\begin{equation*}
			\op{res}_q \op{tr}(f\rho) = 0,\quad f(v)\wedge v\neq 0.
		\end{equation*} 
		Hence, there is $\phi \in H^0(\End_0(V)\otimes K_C)$ such that $\phi_{|q}(v)\wedge v \neq 0$. We can now apply Bertini's theorem to deduce the Theorem.
		  
		If $\rho$ is nilpotent, the fiber $L_{|q}$ is a base point: With the help of Lemma \ref{lemma:local_twisted} we see that the image of 
		\begin{equation*}
			H^0(\End_0(V) \otimes K_C) \to H^0(\End_0(V)\otimes K_C\otimes \OO_q)
		\end{equation*}
		consists of all $f: V_{|q} \to V(q)_{|q}$ which have $L_{|q}$ as an eigenspace. In particular there exists a $\phi: V\to V\otimes K_C$ such that $\phi_{|q}$ has two different eigenspaces. This implies that $V(\tau(\phi))$ is smooth at the point $L_{|q} \in \PP(V)$. Again we can apply Bertini's theorem.
	\end{proof} 
	\section{Unique bad line subbundles}\label{section:unique_bad}
	The goal of this section is to prove Theorem \ref{thm:B} by dimension counting. Consider a line bundle $L$ of degree $s$ and let $t \in\NN$. The technical part of the proof consists of finding an upper bound for the dimension of the following locus of divisors:
	\begin{equation*}
		\cA_{L,t} := \left\{ D \in C_t \mid h^0(K_C \otimes L(-2D)) > 0\right\}.
	\end{equation*}
	Before we can bound the dimension $\dim \cA_{L,t}$ we need to prove two lemmas:
	\begin{lemma}\label{lemma:bad_i}
		Let $M$ be a line bundle on $C$ and $e < \deg M$.  Then, we have:
		\begin{equation*}
			\dim \left\{ \OO(E) \in W^0_e(C) \mid h^0(M(-E)) > 0 \right\} \leq h^0(M) -1.
		\end{equation*}
	\end{lemma}
	\begin{proof}
		We consider the following incidence correspondence:
		\begin{equation*}
			\begin{tikzcd}
			 	& {\{ (E,F) \in C_e \times |M| \mid E\leq F \} }\ar[dr,"{(E,F) \mapsto F}"] \ar[dl,"{(E,F) \mapsto \OO(E)}"] &  \\
				\op{Pic}^e(C)& &{|M|}
			\end{tikzcd}
		\end{equation*}
		The map going to the right is finite. Hence, the dimension of the incidence correspondence is $\leq \dim |M| = h^0(M) -1$. The image of the map going to the left is $\left\{ \OO(E) \in W^0_e(C) \mid h^0(M(-E)) > 0 \right\}$ and the dimension estimate follows.
	\end{proof}
	We are also going to need this auxiliary stratification of the symmetric product $C_t$: 
	\begin{equation*}
		Q_t^k = \left\{D \in C_t \mid h^0(K_C^2(-2D)) = 3(g-1) -2t +k\right\}.
	\end{equation*}
	for $t \leq 2g-2$ and $k \in \NN$.
	\begin{lemma}\label{lemma:bad_ii}
		Assume $t \leq \frac{3}{2}(g-1)$. If $1 \leq k \leq t+2 -g$, then the dimension of $Q^k_{t}$ is at most $t-k$. If $ k > t +2 -g$, then $Q^k_{t}$ is empty.
	\end{lemma}
	\begin{proof}
			If we only care about the linear equivalence class of $D$, we need to consider the following locus:
			\begin{equation*}
				R_t^k = \left\{\OO(D) \in W^0_t(C) \mid h^0(K_C^2(-2D)) = 3g-3-2t+k\right\}.
			\end{equation*} 
			This algebraic set comes with a quasi-finite map to $W_{4(g-1)-2t}^{3(g-1)-1-2t+k}(C)$: We just send $\OO(D)$ to $K_C^2(-2D)$. The fiber over $M \in W_{4(g-1)-2t}^{3(g-1)-1-2t+k}(C)$ is a subset of the square roots of $K_C^2\otimes M^\vee$. Hence, the map has finite fibers. By Serre duality we have $W_{4(g-1)-2t}^{3(g-1)-1-2t+k}(C) \cong W^{k-1}_{2t-2(g-1)}(C)$. If $k > t+2-g$, then this Brill-Noether locus is empty by Clifford's theorem. This implies $Q_t^k = \emptyset$. If $k \leq t+2-g$, Marten's theorem \cite[Theorem 1]{martens1967} implies
			\begin{equation}\label{eq:base}
				\dim R_t^k\leq \dim W^{k-1}_{2t-2(g-1)}(C)\leq 2t-2(g-1) -2(k-1).
			\end{equation}
			
			 Furthermore, there is a natural map $Q_t^k \to R_t^k$. The fiber over $\OO(D) \in R_t^k$ is the complete linear system $|D|$. Let us estimate its dimension $r(D)$ for $k > 1$:
			 \begin{align}\label{eq:fiber}
			 	r(D) &= t -(g-1) + r(K_C(-D)) \\
			 	&\leq t - (g-1) + \frac{1}{2}r(K_C^2(-2D)) \\
			 	&= t-(g-1) +\frac{1	}{2}(3(g-1)-1-2t+k).
			 \end{align}
			  
		In the second line we used the general fact that $r(E) \leq \frac{1}{2}r(2E)$ holds for every divisor. This is a byproduct of the proof of Clifford's theorem. See for example \cite[III, §1]{acgh}. If $k = 1$, we need to have $r(D) = 0$. We now have an estimate for the dimension of every fiber of $Q_t^k \to R_t^k$ and of the the base. In total we get from (\ref{eq:fiber}) and (\ref{eq:base}) for $k > 1$:
		\begin{align*}
			\dim Q_t^k &\leq t-(g-1) +\frac{1	}{2}(3(g-1)-1-2t+k) + 2t-2(g-1) -2(k-1) \\
			&= 2t - \frac{3}{2}(g-1) - \frac{3}{2}k + \frac{3}{2}\\
			&= t - \frac{3}{2}k + \frac{3}{2}.
		\end{align*}
		For $k = 1$, one computes that $\dim Q_t^k \leq 2t - 2(g-1)$. Both estimates imply $\dim Q_t^k \leq t-k$.
	\end{proof}
	\begin{lemma}
		Assume $ 0 < s \leq 2g-2$ and $0 < t \leq g-1 + \frac{s}{2}$. There exists $L \in \op{Pic}^{s}(C)$ such that $K_C\otimes L^\vee$ is effective and 
		\begin{equation}\label{eq:dim_ineq}
			\dim\cA_{L,t} \leq g-2 + s-t
		\end{equation}
		holds. 
	\end{lemma}
	\begin{proof} We use the following incidence correspondence. 
		\begin{equation*}
			\cA_{s,t} := \{ (\OO(E),D) \in W^0_{2g-2 -s}(C) \times C_t \mid h^0(K_C^2(-E-2D)) > 0\}.
		\end{equation*}
		We have a projection $p$ sending $(\OO(E),D)$ to $K_C(-E)$. Its image consists of all $L\in \op{Pic}^s(C)$ such that $h^0(K_C\otimes L^\vee) > 0$. The fiber of $p$ over $L \in W^{1-g+s}_s(C)$ is $\cA_{L,t}$. Hence, we need to compute the dimension of $\cA_{s,t}$.
		We split this computation into two cases: First we consider $s \leq g-2$. Then, $K_C\otimes L^\vee$ is effective for all $L \in \op{Pic}^s(C)$. We have 
		\begin{equation*}
			\cA_{s,t} \cong W^0_{2g-2+s-2t} \times C_t.
		\end{equation*} 
		The isomorphism is given by sending $(\OO(E),D) \in \cA_{s,t}$ to $K_C(-2D-E)$. Therefore, for a general $L \in W_s^{1-g+s}(C)$ we have
		\begin{equation*}
			\dim p^{-1}(L) = \dim \cA_{s,t} - \dim \op{Pic}^s(C) = g-2+s-t.
		\end{equation*} 
	
	Now we assume $s \geq g-1$.	
	 Consider the second projection $p_2: \cA_{s,t} \to C_t$. Assume $D \in Q_t^k$. Then we can apply Lemma \ref{lemma:bad_i} with $M = K_C^2(-2D)$ to deduce
	\begin{equation*}
		\dim p_2^{-1}(D) = 3g-4 - 2t +k.
	\end{equation*}
	We can apply Lemma \ref{lemma:bad_ii} because $t \leq g-1+\frac{s}{2}\leq \frac{3}{2}(g-1)$. Therefore, the inverse image $p_2^{-1}(Q^k_t)$ has dimension $\leq 3g-4 -t$. We conclude that $\dim \cA_{s,t} \leq 3g-4-t$. For a general $L \in W^{1-g+s}_s(C)$ we therefore have: 
	\begin{equation*}
		\dim p^{-1}(L) = \dim \cA_{s,t} - \dim W^{1-g+s}_s(C) = g-2+s-t.
	\end{equation*}
	This is what we wanted to show.
	\end{proof}
	\begin{proof} (of Theorem \ref{thm:B})
		Recall from the introduction that any vector bundle $W$ in $\cW_k^0$ has the following form: There is $M \in Z_k$ such that $W$is an extension of $M^\vee \otimes \Lambda$ by $M$. Note that the bad line subbundles of $W$ are in bijection with bad line subbundles of $W \otimes M^\vee$. Furthermore there is this étale covering of degree $2^{2g}$:
		\begin{equation*}
			Z_k \to \left\{ L \in \op{Pic}^{2k-2+\lambda}(C) \mid h^0(L^\vee\otimes K_C) > 0\right\},\quad M \mapsto M^2 \otimes \Lambda^\vee\otimes K_C.
		\end{equation*} 
		Therefore, we can prescribe $L = M^2 \otimes \Lambda^\vee \otimes K_C$ first
		and we are reduced to the following situation: Take a general $L \in \op{Pic}^{2k-2+\lambda}(C)$ such that $h^0(L^\vee \otimes K_C) > 0$. Consider a general extension $\xi$ of $L$ by $\OO_C$ and its induced vector bundle $V$. We need to show that $V$ has the unique bad line subbundle $\OO_C$. We abbreviate $\deg L$ by $s = 2k-2 +\lambda$. It is in the range $0< s \leq 2g-2- 
		\lambda$. 
		
	    Assume that $M$ is another bad line subbundle of $V$.
		We note that there exists a nonzero morphism $M \to L$. Hence $L$ has the form $M \cong L(-D)$ for some effective divisor $D$. By construction we have the following diagram:
		\begin{equation*}
			\begin{tikzcd}
				0 \ar[r] & \OO_C \ar[r] \ar[d,"="] &\OO_C\oplus L(-D) \ar[r] \ar[d,hook] & L(-D) \ar[r] \ar[d,hook] & 0 \\
				0 \ar[r] & \OO_C \ar[r] &V  \ar[r] & L \ar[r] & 0.
			\end{tikzcd}
		\end{equation*}
		 Hence $\xi$ is in the kernel of 
		\begin{equation*}
			H^1(L^\vee) = \op{Ext}^1(L,\OO_C) \to \op{Ext}^1(L(-D),\OO_C) = H^1(L^\vee(D)).
		\end{equation*}
		For $L(-D)$ to be a bad line subbundle we need to have $\op{Hom}(V/L(-D),L(-D)\otimes K_C) \neq 0$. This is equivalent to $D \in \cA_{L,t}$. Furthermore, we need to have $t \leq g-1 + \frac{s}{2}$ in order for $K_C \otimes L(-2D)$ to have nonnegative degree.
		We now consider:
		\begin{align*}
			\mathcal{B}_{L,t} := \left\{(D ,\xi)\in \cA_{L,t}\times H^1(L^\vee) \mid  \xi \in \op{ker}\left( H^1(L^\vee) \to H^1(L^\vee(D))\right)\right\}
		\end{align*} 
		Projection to the second factor gives a map $\mathcal{B}_{L,t} \to H^1(L^\vee)$. Its image contains all extensions having a second bad line subbundle. 
		The space $\mathcal{B}_{L,t}$ maps to $\cA_{L,t}$ and the fiber of $D$ is $\ker\left(H^1(L^\vee) \to H^1(L^\vee(D))\right)$. We have
		\begin{equation*}
			\dim \ker\left(H^1(L^\vee) \to H^1(L^\vee(D))\right) = \begin{cases}
				t -1, &L(-D) \cong \OO_C, \\
				t , &\text{else.}
			\end{cases}
		\end{equation*}
		Hence we have $\dim \mathcal{B}_{L,t} \leq \dim \cA_{L,t} +t \leq g-2 + s < h^1(L^\vee)$. This is what we wanted to show.
	\end{proof}
	For the application in the next section we will need the following refinement of the last result:
	\begin{lemma}\label{lemma:general_wobbly_bundle}
		Assume $\lceil\frac{g-\lambda}{2}\rceil \leq k \leq g-\lambda$. Then there exists a vector bundle $V$ in $\cW_k$ satisfying the following properties:
		\begin{itemize}
			\item[(1)] $V$ has a unique bad line subbundle $L_1$ of degree $1-k$. Write $L_2 := V/L_1$ and let $\xi$ be the extension of $L_2$ by $L_1$ induced by $V$.
			\item[(2)] The space $H^0(L_1\otimes L_2^\vee\otimes K_C)$ is generated by some $\phi_{12}$. 
			\item[(3)] There is no solution of the following system of equations:
			\begin{equation}\label{eq:weird_cond}
				2\phi_{12}\phi_{21}+\phi_{11}^2 = 0, \quad \phi_{11}\cup \xi = \phi_{21} \cup \xi = 0,
			\end{equation}
			with $\phi_{21}\in H^0(L_1^\vee\otimes L_2\otimes K_c), \phi_{11}\in H^0(K_C)$.
		\end{itemize}
	\end{lemma}
	\begin{proof}
		We use the notation from the proof of Theorem \ref{thm:B} . Fix $L \in W^0_{\lambda - 2 + 2k}(C)$ such that $\dim \cA_{L,t}$ satisfies inequality (\ref{eq:dim_ineq}) for all $t$ in the range $1\leq t \leq g-1 + \frac{\deg L}{2}$. We also assume $H^0(L)$ is one dimensional and generated by $\phi_{12}$ such that $\div \phi_{12}$ is reduced.
		Choose a line bundle $L_1$ such that $L_1^2 \otimes L \cong \Lambda$ and put $L_2 = L_1\otimes L$. Note that the degree of $L_1$ is $1-k$. We abbreviate the degree of $L_i$ by $d_i$ for $i = 1,2$. The proof of the last Lemma shows that a general extension $\xi \in \op{Ext}^1(L_2,L_1)$ produces a vector bundle whose only bad line subbundle is $L_1$. 
		
		We now show that we can choose $\xi$ such that (3) is satisfied. We first reformulate (\ref{eq:weird_cond}). Assume we have a solution $(\phi_{21},\phi_{11})$ for some $\xi$. 
		Then every zero of $\phi_{12}$ is also a zero of $\phi_{11}$. Because $\div \phi_{12}$ is reduced, we conclude that there exists $s\in H^0(L_1^\vee\otimes L_2)$ such that $\phi_{11} = \phi_{12}s$. Therefore, the following condition implies (3):
		\begin{equation*}
			\phi_{12}s \cup \xi \neq 0\quad \text{for all $s \in H^0(L_1^\vee\otimes L_2)$}.
		\end{equation*}
		It suffices to show that the following algebraic subset of $H^1(L_1\otimes L_2^\vee)$ is a proper subset:
		\begin{equation*}
			\cB := \bigcup_{[s] \in \PP H^0(L_1^\vee \otimes L_2)} \ker\left( H^1(L_1\otimes L_2^\vee) \xrightarrow{\phi_{12}s\cdot} H^1(L_1^\vee\otimes L_2 \otimes K_C)\right).
		\end{equation*}
		We give an upper bound on the dimension of $\cB$. First observe that 
		\begin{equation*}
			\dim \ker(H^1(L_1\otimes L_2^\vee) \xrightarrow{\phi_{12}s\cdot} H^1(L_1^\vee\otimes L_2 \otimes K_C)) = \begin{cases}
				2g-2, & d_2 - d_1 = g-2,\\
				2g-3, &d_2 - d_1 > g-2. 
			\end{cases}
		\end{equation*} 
		If $d_2 - d_1 = g-2$ we have $h^0(L_1^\vee\otimes L_2) = 1-g + d_2 -d_1$ and if $d_2 - d_1 > g-2$ we have $h^0(L_1^\vee\otimes L_2) = 2-g+d_2 -d_1$. In both cases it follows that $\dim \cB \leq h^1(L_1\otimes L_2^\vee) -1$.
		\end{proof}
	\section{The restricted Hitchin map of a general wobbly vector bundle}\label{section:wobbly}
	Now we prove the first part of Theorem \ref{thm:C}: A general wobbly vector bundle $V  \in \cW_k$ has a dominant restricted Hitchin map. Recall that we assume $\lceil\frac{g-\lambda}{2}\rceil \leq k \leq g-\lambda$. A key input will be the result of the last section that a general wobbly vector bundle has exactly one bad line subbundle.  
	\begin{proof} (of Theorem \ref{thm:C} (i))
	Take a vector bundle $V$ satisfying the conditions of Lemma \ref{lemma:general_wobbly_bundle}. Let $\phi$ be a nilpotent twisted endomorphism of $V$. By our choice of $V$ it is unique up to a scalar.  We will denote $\PP H^0(\End_0(V) \otimes K_C)$ by $\PP \cE$.
	We will denote the projectivization of the restricted Hitchin map by $h$.
	 The indeterminacy locus of $h$ is $\{[\phi]\}$. Let us blow up $[\phi]$ in $\PP\mathcal{E}$. We have a projection $\pi_1: \op{Bl}_{[\phi]}\PP \cE  \to \PP \cE$. We can extend $h$ to $h': \op{Bl}_{[\phi]}\PP\cE \dashrightarrow \PP H^0(K_C^2)$ such that the codimension of the indeterminacy locus of $h''$ has codimension atleast $2$. We have $(h')^*\OO(1) \cong \OO(2H - E_1)$, where $H$ is the pullback of the hyperplane class and $E_1$ is the exceptional divisor. For the latter divisor we have
	\begin{equation*}
		E_1 \cong \PP \left( H^0(\End_0(V)\otimes K_C)/\langle \phi\rangle \right).
	\end{equation*}
	The map on the exceptional divisor is the projectivization of the differential of $h$: Let $\phi_{12} \in H^0(K_C\otimes L^\vee)$ be the section inducing $\phi$ and
	 $[\chi]$ be in $E_1$. Then (\ref{eq:diff_nilpotent}) shows:
	\begin{equation*}
		h'([\chi]) = [d_\phi h_V(\chi)] = [\phi_{12}c_\xi(\chi)].
	\end{equation*}
	 By Lemma \ref{lemma:ker_diff} the indeterminacy locus of $h'$ is given by 
	\begin{equation*}
	M = \PP \left( \ker(c_\xi)/\langle \phi\rangle \right) \subseteq E_1.
	\end{equation*} 
	Using Lemma \ref{lemma:image_c} one can quickly compute that $M$ is empty if and only if $k = g, \lambda = 0$. In this case one can prove that $\OO(2H-E_1)$ is ample. Therefore, $h'$ is finite. 
	
	Let us now treat the remaining cases.
	We blow up $M$ and obtain a rational map $h'': \op{Bl}_{M}\op{Bl}_{[\phi]}\PP \mathcal{E} \dashrightarrow \PP H^0(K_C^2)$. Denote the projection $\op{Bl}_{M}\op{Bl}_{[\phi]} \PP \cE \to  \op{Bl}_{[\phi]}\PP\cE$ by $\pi_2$. Call its exceptional divisor $E_2$. We have $(h'')^*\OO(1) \cong \OO(2H-\pi_2^*E_1 -E_2)$. We show that it is in fact a morphism.  The normal space of $M$ at a point $[\chi]\in M$ is isomorphic to 
	\begin{equation*}
		N_{M}\op{Bl}_{[\phi]}\PP \cE_{|[\chi]} \cong \langle \chi \rangle \oplus \frac{H^0(\End_0(V)\otimes K_C)}{\op{im}(c_\xi)}.
	\end{equation*}
	We can use this to describe $h''$ on the exceptional divisor $E_2$:  
	\begin{equation}\label{eq:diff_h''}
		h''(([\chi],[\chi',\psi])) = [2\phi_{12}c_\xi(\psi) + \chi_{11}\chi'_{11}].
	\end{equation}
	Here $[\chi] \in E_1$ and $(\chi',\psi)$ is in $N_{M}\op{Bl}_{[\phi]}\PP \cE_{|[\chi]}$. Furthermore, we write $\chi_{11} := \chi_{|L_1} \in H^0(K_C)$. To prove (\ref{eq:diff_h''}) observe that the map $\Aff^1\setminus \{0\} \to \PP \cE, z\mapsto [\phi+ z\chi + z^2\psi]$ can be extended to a map 
	 $\Aff^1 \to \op{Bl}_{M}\op{Bl}_{[\phi]}\PP\cE$ which maps $0$ to $([\chi],[\chi,\psi])\in E_2$. Then we compute:
	\begin{equation*}
	[\op{tr}\left( (\phi+ z\chi + z^2\psi)^2 \right)] =  [z^2\op{tr}(2\phi\psi+\chi^2)] = [2\phi_{12}c_{\xi}(\psi)+\chi_{11}^2].
	\end{equation*}
	Furthermore, we observe that $\chi_{11}$ satisfies $\chi_{11}\cup \xi=0$. To see this use the following long exact sequence:
	\begin{align*}
		0 &\to H^0(L_2^\vee \otimes L_1 \otimes K_C) \to H^0(V^\vee\otimes L_1\otimes K_C) \to H^0(K_C) \\
		&\xrightarrow{ (-)\cup \xi} H^1(L_2^\vee\otimes L_1 \otimes K_C) \to H^1(V^\vee\otimes L_1\otimes K_C) \to H^1(K_C) \to 0.
	\end{align*}
	Recall from Lemma \ref{lemma:image_c} that $ c_\xi(\phi) \cup \xi= 0$. We see that any base point of $h''$ gives a non-trivial solution to (\ref{eq:weird_cond}). Therefore, $h''$ is base-point free. 
	
	Lemma \ref{lemma:linear_system} is going to show that the map
	\begin{equation*}
		\phi_{2H-\pi_2^*E_1-E_2}: \PP \cE \to \PP H^0(\OO(2H-\pi_2^*E_1-E_M))^\vee
	\end{equation*}
	 is injective on $(\pi_2\circ \pi_1)^{-1}(\PP\cE\setminus \PP(\ker c_\xi))$. Denote by $\pi$ the projection $\PP H^0(\OO(2H-\pi_2^*E_1-E_M))^\vee \dashrightarrow \PP H^0(K_C^2)$. We can write:
	\begin{equation*}
		h''= \pi_{|\op{im}\varphi_{2H-\pi_2^*E_1-E_2}}\circ\varphi_{2H-\pi_2^*E_1-E_2}.
	\end{equation*}
	The map $\pi_{|\op{im}\varphi_{2H-\pi_2^*E_1-E_2}}$ is finite because it is a projection. This implies that the restriction of $h$  $\PP\cE\setminus \PP(\ker c_\xi)$ has finite fibers.
	 We now take $\eta \in H^0(K_C^2)$ defining an irreducible spectral curve. Let $\psi$ be a twisted endomorphism with $h_V(\psi) = \eta$. From Lemma \ref{lemma:reducible} we deduce $\psi \notin \ker c_{\xi}$ . We have shown that $h$ is finite on $\PP \cE\setminus \PP \ker_\xi$, hence  $\psi$ is an isolated point of $h_V^{-1}(\eta)$. 
\end{proof}
\section{Reducible curves}\label{section:reducible}
Fix a vector bundle $V$ of rank $2$ over $C$ whose determinant we denote by $\Lambda$. Consider reducible spectral curves of degree $2$ over $C$. They are of the form $\tilde{C} = V(\tau^2 - \omega^2)$ where $\omega$ is a nonzero holomorphic differential on $C$. We want to describe the fiber of $h_V$ over $\omega^2$. In other words we want to find all line bundles $M$ on $\tilde{C}$ whose pushforward is isomorphic to $V$. This is achieved by the following proposition. Before we can state it let us introduce some notation. Consider the exact sequence:
\begin{equation*}
	0 \to L \to L\otimes K_C \to L\otimes K_{C|\op{div}\omega} \to 0.
\end{equation*} 
We then define the map $\delta$ as the connection map of the following long exact sequence:
\begin{align}
	0 &\to \op{Hom}(\Lambda\otimes L^\vee,L) \to \op{Hom}(\Lambda\otimes L^\vee,L\otimes K_C) \to \op{Hom}(\Lambda\otimes L^\vee,L\otimes K_{C|\op{div}\omega}) \\
	&\xrightarrow{\delta} \op{Ext}^1 (\Lambda\otimes L^\vee,L) \to \op{Ext}^1(\Lambda\otimes L^\vee,L\otimes K_C) \to 0\label{eq:long_exact_delta}
\end{align}
\begin{proposition}\label{prop:reducible}
	Let $V$ be a vector bundle of rank $2$ on $C$. Write $\Lambda = \det V$. Then $h_V^{-1}(\omega^2)$ is in bijection with the following set:
	\begin{equation*}
		\left\{ (i:L \hookrightarrow V, \tilde{\xi}) \;\middle|\; \begin{array}{l}
			\text{$L$ is a line subbundle of $V$} \\
			\text{inducing the extension class $\xi \in \op{Ext}^1 (\Lambda\otimes L^\vee,L) $}, \\
			\text{$\tilde{\xi} : \Lambda \otimes L^\vee_{|\op{div}\omega} \xrightarrow{\sim} L\otimes K_{C|\op{div}\omega}$, such that $\delta(\tilde{\xi}) = \xi$.}
		\end{array} \right\}/\sim
	\end{equation*}
	Here we are considering pairs $(L,\tilde{\xi})$ and $(L',\tilde{\xi}')$ as equivalent if $L \cong L'$ and $\tilde{\xi}' = \lambda \tilde{\xi}$ for some $\lambda \in \CC^\times$.  
\end{proposition}
\begin{proof}
	We start with a line bundle $M$ on $\tilde{C}$. The spectral curve is the union of two copies of $C$ which we call $C_1$ and $C_2$. We define 
	\begin{equation*}
		M_1 := M_{|C_1}, M_2 := M_{|C_2}.
	\end{equation*} 
	Then we have an embedding $i: M_1  \otimes K_C^\vee \hookrightarrow \psi_*M$ because we can regard any section of $M_1\otimes K_C^\vee$ as a section of $M_1$ on $C_1$ which vanishes at the points of $C_1$ connecting it to $C_2$. Hence we can extend it by zero to obtain a section of $M$. We then put $L = M_1\otimes K_C^\vee$. We also have $M_2 \cong \Lambda \otimes L^{\vee}$. Restricting a section of $M$ to $C_2$ gives a map $\psi_*\to \Lambda \otimes L^\vee$. Hence, $\psi_*M$ is an extension $\xi$ of $L$ by $\Lambda \otimes L^{\vee}$.
	Let us describe its extension class. A line bundle on $\tilde{C}$ is determined by a triple $(M_1,M_2,\tilde{\xi})$, where $\tilde{\xi}$ is an isomorphism $M_{2|\div\omega} \to M_{1|\div\omega}$. We can describe $M$ locally around the singularity corresponding to $p \in \op{supp}(\omega)$ as follows: Assume $\omega$ has multiplicity $m$ around $p$.  
	Then an element of the stalk $M\otimes \OO_{\tilde{C},p}$ is a pair $(s_1,s_2)$, where $s_i \in M_i \otimes \OO_{C,p}$ for $i=1,2$ such that
	\begin{equation*}
		s_{1|mp} = \tilde{\xi}(s_{2|mp}).
	\end{equation*}
	One can then gets the following description of $\psi_*M$:
	\begin{equation*}
		\psi_*M \cong \ker\left(M_1 \oplus M_2 \xrightarrow{(\op{ev}_{\div \omega}, -\tilde{\xi})} M_{1|\op{div}\omega}\right)
	\end{equation*}
	By \cite[Lemma 4.10]{nollau} the extension class is given by $\delta(\tilde{\xi})$.
\end{proof}
\begin{remark}
	Assume $V$ is a wobbly vector bundle. Then any $\phi \in \op{\ker} c_\xi$ has a reducible spectral curve by Lemma \ref{lemma:reducible}. By (\ref{eq:long_exact_delta}) the map $\delta$ has a nontrivial kernel which implies that $h_V^{-1}(\omega^2)$ is positive dimensional for every $\omega \in H^0(K_C)$. We see that the restricted Hitchin map of a wobbly vector bundle has many positive dimensional fibers. 
\end{remark}
\begin{corollary}
	Let $V$ be a vector bundle of rank $2$ on $C$. Write $\Lambda = \det V$ and assume $\deg \Lambda = \lambda \in \{0,1\}$. Assume $L$ is a subbundle of $V$ such that $h^1(\Lambda^\vee \otimes L^2 \otimes K_C) = 0 $. Then $\{ \omega^2 \mid \omega \in H^0(K_C)\}$ is in the image of $h_V$.
\end{corollary}
\begin{proof}
	Take any $\omega \in H^0(K_C)$. By $(\ref{eq:long_exact_delta})$ the cokernel of $\delta$ is isomorphic to $H^1(\Lambda^\vee\otimes L^2\otimes K_C) = 0$. Hence the extension $\xi$ corresponding to $V$ is in the image of $\delta$. By (\ref{prop:reducible}) $\omega^2$ is in the image of $h_V$.
\end{proof}
\section{Analysis of the blowup}\label{section:blowup}
In this section we investigate the blowup considered in section \ref{section:wobbly}. To make our arguments more transparent we do this in slightly more generality: Consider $\PP^N$, a linear subspace $M_0$ of $\PP^N$ of codimension $m$ and $p\in M_0$. Assume $2 \leq m \leq N-1$. We write $X_1 = \op{Bl}_p\PP^N$ and $\pi_1$ for the natural map $X_1 \to \PP^N$. Call its exceptional divisor $E_1$. Then $M:= E_1 \cap \pi_1^{-1}(M_0)$ is a hyperplane in $E_1$ of codimension $m$. Denote the blowup $\op{Bl}_MX_1$ by $X_2$ and its map to $X_1$ by $\pi_2$. Let $E_2$ be its exceptional divisor.

\begin{lemma}\label{lemma:linear_system}
	The map $\varphi_{2H-\pi_2^*E_1-E_2}: X_2 \to \PP H^0(\OO(2H-\pi_2^*E_1 - E_2))^\vee$ is injective on $X_2 \setminus (\pi_2\circ \pi_1)^{-1}(M_0)$. 
 \end{lemma}
\begin{proof}
	We introduce homogeneous coordinates $x_0,\ldots,x_N$ such that 
	\begin{align*}
		\{p\} &= V(x_1,\ldots,x_N),\\
		M_0 &= V(x_{N-m+1},\ldots,x_N).
	\end{align*} 
	The blowup $X_1$ is a closed subscheme of $\PP^N \times \PP^{N-1}$ defined by
	\begin{equation*}
		x_iy_j = x_jy_i,\quad i,j\in \{1,\ldots,N\}.
	\end{equation*} 
	Here $y_1,\ldots,y_N$ are homogeneous coordinates of $\PP^{N-1}$. Notice that $M$ can be represented as a zero locus in two ways: 
	\begin{align*}
		M &= V(x_1,\ldots,x_N, y_{N-m+1},\ldots,y_N)\\
		&= V\left(\{x_{i}y_j\}_{i,j \in \{1,\ldots,N\}}, \quad \{x_0y_j\}_{ j \in \{N-m+1,\ldots,N\}}\right).
	\end{align*}
	The second set of equations has the advantage that all elements are sections of $\OO(2H-E_1)$. We obtain the global sections of $\OO(2H - \pi_2^*E_1 - E_2)$:
	\begin{equation*}
		H^0(X_2,\OO(2H-\pi_2^*E_1-E_2)) \cong H^0X_1,\OO(2H-E_1)\otimes \mathcal{I}_{M}).
	\end{equation*}
	The morphism $\varphi_{2H-\pi_2^*E_1-E_2}$ restricted to the complement of $E_2$ coincides with 
	\begin{equation}\label{eq:blowup_sections}
		f = [\{x_{i}y_j\}_{i,j \in \{1,\ldots,N\}}, \quad \{x_0y_j\}_{ j \in \{N-m+1,\ldots,N\}}] : X_1 \setminus \pi_1^{-1}(M_0) \to \PP^{N^2+m-1} 
	\end{equation}
	because the sections defining $f$ are a basis of $H^0(X_1,\OO(2H-E_1)\otimes \mathcal{I}_{M})$. The map $f$ is injective on $X_1 \setminus \pi_1^{-1}(M_0)$: Take $q,r$ to be points in this open set. There exists $l \in \langle y_{N-m+1},\ldots,y_N\rangle$ such that $l(q) \neq 0$ and $l(r) \neq 0$. Then we can use the sections $(x_il)_{i=0}^n$ to separate $r$ and $q$.  
\end{proof}
\begin{lemma}
	The degree of $\varphi_{2H-\pi_2^*E_1-E_2}: X_2 \to \PP H^0(\OO(2H-\pi_2^*E_1 - E_2))^\vee$ is given by
	\begin{equation*}
		\deg \varphi_{2H-\pi_2^*E_1-E_2} = 2^N- 2^{N-m}.
	\end{equation*}
\end{lemma}
\begin{proof}
	The degree is given by the following intersection number: 
	\begin{equation*}
		\deg h'' = \int_{X_2} (2H+\pi_2^*E_1+E_2)^{N}.
	\end{equation*}
	The exceptional divisors $E_1$ and $E_2$ are both the projectivization of the respective normal bundles. Denote by $\zeta_i$ the class of $\OO_{E_i}(1)$ in the Chow ring of $E_i$. We will use a few standard identities in the Chow ring of a blowup. They can be found in \cite[Proposition 13.12]{eisenbud_harris}. We will denote the embedding of $E_i$ into $X_i$ by $j_i$. For $\alpha \in A^{\bullet}(X_1), \beta, \gamma \in A^{\bullet}(E_2)$:
	\begin{align}
		\pi_2^*(\alpha) j_{2*}(\beta) &= j_{2*}(\alpha'\cdot\beta)\label{eq:identity_a} \\
		j_{2*}(\beta)j_{2*}(\gamma) &= j_{2*}(\beta\gamma\zeta)\label{eq:identity_b}.
	\end{align}
	Here $\alpha'$ is the pullback of $\alpha$ along $E_2 \to M \to X_1$. As a consequence of (\ref{eq:identity_b}) we find:
	\begin{equation}\label{eq:identity_c}
		E_2^k = j_{2*}(1)^k =(-1)^{k-1}j_{2*}(\zeta_2^{k-1}).
	\end{equation} In our calculation we will need the Chow polynomial of the normal bundle $N_MX_1$. It sits in the following normal sequence:
	\begin{equation*}
		0 \to N_ME_1 \to N_MX_1 \to N_{E_1}X_1 \to 0.
	\end{equation*}
	By the Whitney sum formula we get:
	\begin{equation*}
		c_t(N_MX_1) = c_t(N_ME_1)c_t(N_{E_1}X_1) = (t + \zeta_{1|M})^m(t - \zeta_{1|M}).
	\end{equation*}
	Using \cite[Theorem 9.6]{eisenbud_harris} we can write down a presentation of $E_2 \cong \PP_M\left( N_MX_1\right)$:
	\begin{align*}
		A^\bullet(E_2) &= A^\bullet(M)[\zeta_2]/\left((t + \zeta_{1|M})^m(t - \zeta_{1|M})\right)\\
		&= \ZZ[H_2,\zeta_2]/\left(H_2^{N-m},(\zeta_2+H_2)^m(\zeta_2 -H_2 )\right).
	\end{align*}
	Here $H_2$ is the pullback of $\zeta_{1|M}$ to $E_2$. 
	As a last preparation we note that $H\cdot \pi_2^*E_1 = H\cdot E_2 = 0$ because the restriction of the hyperplane class in $\PP\cE$ to the point $[\phi]$ is trivial. Now we are in a position to compute:
	\begin{align*}
		(H-\pi_2^*E_1-E_2)^N &= H^N + (-1)^N(\pi^*_2E_1+E_2)^N.
	\end{align*} 
	We inspect the second term:
	\begin{align*}
		(\pi^*_2E_1 +E_2)^N&= \pi_2^*E_1^N + \sum_{k=1}^N \binom{N}{k}\pi_2^*(E_1^{N-k})E_2^k \\
		&= \pi_2^*E_1^N + \sum_{k=1}^{N}\binom{N}{k}\pi_2^*(E_1^{N-k})j_{2*}((-1)^{k-1}\zeta_2^{k-1})\qquad (\ref{eq:identity_c})\\
		&= \pi^*E_1^N + (-1)^{N-1}j_{2*}\left(\sum_{k=1}^{N}H_2^{N-k} \zeta_2^{k-1}\right)\qquad(\ref{eq:identity_a}).
	\end{align*} 
	We now use the identity $\sum_{k=1}^N\binom{N}{k}a^{N-k}b^{k-1} = \sum_{l=0}^{N-1}(a+b)^{N-1-l}b^l$ to deduce:
	\begin{align*}
		\sum_{k=1}^{N}H_2^{N-k} \zeta_2^{k-1} &= \sum_{l=0}^{N-1} (H_2 + \zeta_2)^{N-1-l}H_2^l \\
		&=\sum_{l=1}^{N-m-1} (H_2 + \zeta_2)^m(H_2+\zeta_2)^{N-m-1-l}H_2^l \\
		&= \sum_{l=0}^{N-m-1} (H_2 + \zeta_2)^m 2^{N-m-1-l}H_2^{N-m-1}\\
		&= \sum_{l=0}^{N-m-1}2^{N-m-1-l}\zeta_2^mH_2^{N-m-1}\\
		&= (2^{N-m}-1)\zeta_2^mH_2^{N-m-1}.
	\end{align*}
	Let us explain the steps: We first used the identity $H_2^{N-m} = 0$ to only sum until $l = N-m-1$. Then we applied $H_2(H_2 + \zeta_2)^m = \zeta_2 (H_2 + \zeta_2)^m$ to deduce $(H_2 + \zeta_2)^m(H_2+\zeta_2)^{N-m-1-l} = 2^{N-m-1-l}(H_2+\zeta_2)^mH_2^{N-m-1-l}$. In the next step we used that $H_2$ is nilpotent of order $N-m$. Now we observe that $\zeta_2^mH_2^{N-m-1}$ is the fundamental class of $E_2$. Putting everything together we find:
	\begin{align*}
		\int_{X_2}(2H+\pi_2^*E_2+E_1)^N &= \int_{X_2}2^NH^N +(-1)^{N-1} \int_{X_1}E_1^{N} - \int_{E_2} (2^{N-m} - 1)\zeta_2^mH_2^{N-m-1} \\
		&= 2^N + (-1)^N(-1)^{N-1}- (2^{N-m}-1)\\
		&= 2^N - 2^{N-m}.
	\end{align*}
\end{proof}
\begin{proof} (of Theorem \ref{thm:C} (ii))
	 Assume that $V$ is of the same form as in part (i) of this Theorem. If $(k,\lambda) \neq (g,0)$, we saw that we can resolve the rational map $\PP \cE \to \PP H^0(K_C^2)$ by blowing up twice. The degree of $h''$ equals the degree of $\phi_{2H-\pi_2^*E_1-E_2}$ and is given by the previous Lemma. We have $N = 3g-3$ and we find $m$ using Lemma \ref{lemma:image_c} and the Riemann-Roch theorem:
	\begin{align*}
		m &= \op{codim}_{H^0(\End_0(V)\otimes K_C)} \ker(c_\xi)\\
		&= \dim \op{Hom}(L,\det( V) \otimes L^\vee\otimes K_C) -1\\
		&= \deg (L^{-2}\otimes \det V \otimes K_C) + 1-g -1 \\
		&= 2k+g-4 +\lambda.
	\end{align*}
	We are left with the special case $(k,\lambda) = (g,0)$ which is much easier. We showed that blowing up $[\phi]$ suffices to resolve the indeterminacy. The degree of $h' : \op{Bl}_{[\phi]}\PP\cE \to \PP H^0(K_C^2)$ is given by 
	\begin{equation*}
		\int_{X_1} (2H-E_1)^{3g-4}= 2^{3g-4} -1.
	\end{equation*}
To finish the proof we note that $\deg h_V$ is twice the degree of $h: \PP\cE \dashrightarrow \PP H^0(K_C^2)$.
\end{proof}
\printbibliography
\end{document}